\documentclass[11pt]{article}
\usepackage{amsthm, amsmath, amssymb, amsfonts, url, booktabs, tikz, setspace, fancyhdr, enumerate}
\usepackage[margin = 1in]{geometry}

\usepackage[czech,english]{babel} % This will make Stovicek look correct

\usepackage[hidelinks]{hyperref}

\usepackage{cleveref}

\newtheorem{theorem}{Theorem}[section]
\newtheorem{proposition}[theorem]{Proposition}
\newtheorem{lemma}[theorem]{Lemma}

\newtheorem{question}[theorem]{Question}

\theoremstyle{definition}

\theoremstyle{remark}

\newcommand{\HH}{\mathcal{H}}
\newcommand{\GG}{\mathcal{G}}

\newcommand{\mockalph}[1]{}

\tikzstyle{p}+=[fill=black, circle, minimum width = 1pt, inner sep =
1pt]
\tikzstyle{w}+=[fill=white, draw, circle, minimum width = 1pt, inner sep =
1.5pt]

\title{\vspace{-0.7cm}Common graphs with arbitrary connectivity and chromatic number}
\author{Sejin Ko\thanks{Department of Mathematics, 
		Hanyang University, Seoul. %, 
		Email: {\tt tpwls960104@hanyang.ac.kr}.}
\and
Joonkyung Lee\thanks{Department of Mathematics, 
		Yonsei University, Seoul.
		Email: \texttt{joonkyunglee}@\texttt{yonsei.ac.kr}.
		}
}
\date{}

\begin{document}
\maketitle

\begin{abstract}
A graph $H$ is \emph{common} if the number of monochromatic copies of $H$ in a 2-edge-colouring of the complete graph $K_n$ is asymptotically minimised by the random colouring.
We prove that, given $k,r>0$, there exists a $k$-connected common graph with chromatic number at least $r$. The result is built upon the recent breakthrough of Kr\'a\v{l}, Volec, and Wei who obtained common graphs with arbitrarily large chromatic number and answers a question of theirs. 
\end{abstract}

\section{Introduction}
A central concept in graph Ramsey theory is the \emph{Ramsey multiplicity} of a graph $H$, which counts the minimum number of monochromatic copies of $H$ in a 2-edge-colouring of the $n$-vertex complete graph~$K_n$.
There are some graphs $H$, the so-called \emph{common graphs}, such that the number of monochromatic $H$-copies in a 2-edge-colouring of $K_n$ is asymptotically minimised by the random colouring. 
For example, Goodman's formula~\cite{G59} implies that a triangle $K_3$ is common, which is one of the earliest results in the area.

Partly inspired by Goodman's formula, Erd\H{o}s~\cite{E62common} conjectured that every complete graph is common. This was subsequently generalised by Burr and Rosta~\cite{BR80}, who conjectured that every graph is common. In the late 1980s, both conjectures were disproved by Thomason~\cite{Thom89} and by Sidorenko~\cite{Sid89common}, respectively.
Since then, there have been numerous attempts to find new common (or uncommon) graphs, e.g., \cite{F08,JST96,Sid96}. Although the complete classification seems to be still out of reach,
new common graphs have been found during the last decade by using some advances on Sidorenko's conjecture~\cite{Sid93} or the computer-assisted flag algebra method~\cite{R07}.
For more results along these lines, we refer the reader to one of the most recent results~\cite{CL20} on Sidorenko's conjecture and some applications of the flag algebra method~\cite{GLLV22,HHKNR12} with references therein.

Despite all these studies on common graphs, all the known common graphs only had chromatic numbers at most four. This motivated a natural question, appearing in~\cite{CFS15survey,HHKNR12}, to find a common graphs with arbitrarily large chromatic number.
This question remained open until its very recent resolution by Kr\'a\v{l}, Volec, and Wei~\cite{KVW22}.
Since their construction connects a graph with high chromatic number and girth to a copy of a complete bipartite graph by a long path, they asked~\cite[Problem 25]{KVW22} if highly connected common graphs with large chromatic number exist. We answer this question in the affirmative.

\begin{theorem}\label{main}
Let $k$ and $r$ be positive integers. Then there exists a $k$-connected common graph with chromatic number at least $r$. 
\end{theorem}

We remark that this short follow-up note to the recent result only partially presents various in-depth studies on common graphs and relevant questions. For a modern review of a variety of results in the area, we refer the reader to recent articles~\cite{FW22,KVW22}.
%As this is a short follow-up to the recent result, our bibliography only partly covers the rich history of all the studies on common graphs and relevant questions. For a modern review of a variety of results in the area, we refer the reader to~\cite{KVW22}.

\section{Proof of the main theorem}

A useful setting to analyse commonality of graphs is to use the modern theory of dense graph limits~\cite{L12}. A \emph{graphon} is a two-variable symmetric measurable function $W:[0,1]^2\rightarrow [0,1]$ and the \emph{homomorphism density} of a graph $H$ is defined by
\begin{align*}
    t(H,W) := \int \prod_{uv\in E(H)}W(x_u,x_v)~ d\mu^{V(H)},
\end{align*}
where $\mu$ denotes the Lebesgue measure on $[0,1]$.
In this language, a graph~$H$ is common if and only if $t(H,W)+t(H,1-W)\geq 2^{1-e(H)}$ for every graphon $W$, where $e(H)$ denotes the number of edges in $H$.

\medskip

The \emph{$q$-book} $H_I^q$ of~$H$ along an independent set $I\subseteq V(H)$ of $H$ is the graph obtained by taking $q$ vertex-disjoint copies of $H$ and identifying the corresponding vertices in $I$.
The following lemma is a straightforward consequence of Jensen's inequality.
\begin{lemma}\label{Jensen}
Let $I$ be an independent set of a graph $H$.
If $H$ is common, then $H_I^q$ is also common for every positive integer $q$.
\end{lemma}
\begin{proof}
By a standard application of Jensen's inequality, the inequality $t(H_I^q,W) \geq t(H,W)^q$
holds for every graphon $W$. Therefore,
\begin{align*}
    t(H_I^q,W) + t(H_I^q,1-W) & \geq t(H,W)^q + t(H,1-W)^q \\
    &\geq 2\cdot\left(\frac{t(H,W)+t(H,1-W)}{2}\right)^q \\ 
    &\geq 2^{1-q\cdot e(H)} = 2^{1-e(H_I^q)},
\end{align*}
where the second inequality is again by convexity and the last inequality uses commonality of $H$. This proves commonality of $H_I^q$.
\end{proof}
To summarise, commonality is preserved under the $q$-book operation.
Another advantage of the operation is that it preserves chromatic numbers. Indeed, a proper colouring of $H$ can be naturally extended to $H_I^q$ by assigning the same colour as a vertex of $H$ to its `clones' in $H_I^q$. 
Our key idea is to repeatedly apply the $q$-book operation to a common graph $H$, which increases connectivity while maintaining the chromatic number and commonality of $H$. 

First, enumerate the vertices in an $r$-vertex graph $H$ by $V(H)=\{v_1,v_2,\cdots,v_r\}$. Let $H_0:=H$ and let $H_i:=(H_{i-1})_{U_i}^q$, the $q$-book of $H_{i-1}$ along $U_i$, where $U_i$ is the set of all copies of $v_i$ in $H_{i-1}$.
The \emph{$q$-bookpile} $H(q)$ of $H$ is then the graph $H(q):=H_{r}$ after the full $r$-step iteration.
It is not hard to see that this graph $H(q)$ is independent of the initial enumeration and hence well-defined.
%This notation is well-defined since $q$-blow-up is independent of enumeration.
For example, if $H=K_r$ and $q=2$, then $H(2)$ is the line graph of an $r$-dimensional hypercube graph. %$Q_n$. 
This graph in fact appeared in~\cite{CHPS12} in a different context, which partly inspired our approach.
As each $H_i$ decomposes to $q$ edge-disjoint copies of~$H_{i-1}$, the $q$-bookpile $H(q)$ decomposes to $q^r$ edge-disjoint copies of $H$. 
To distinguish these, we say that the $q^r$ edge-disjoint $H$-subgraphs of $H(q)$ as the \emph{standard copies} of $H$ in $H(q)$.

\medskip

%By the fact that $H(q)$ has the same chromatic number as $H$ and~\Cref{Jensen}, $H(q)$ is common. 
%By~\Cref{Jensen}, $H(q)$ is common whenever $H$ is. Furthermore, $H(q)$ has the same chromatic number as $H$.
%and $r$-chromatic whenever $H$ is.
%Therefore, 
As already sketched, the following theorem together with the construction of connected common graphs $H$ with arbitrarily large chromatic number in~\cite{KVW22} implies \Cref{main}:
\begin{theorem}\label{blow_up_connected}
Let $H$ be a connected graph. For every positive integer~$k$, there exists $q=q(k,H)$ such that the $q$-bookpile $H(q)$ of $H$ is $k$-connected.
\end{theorem}

To analyse connectivity of $H(q)$, we consider an auxiliary hypergraph on $V(H(q))$ whose edge set consists of the standard copies of $H$ in $H(q)$.
We shall first describe what this hypergraph looks like.

Write $[q]:=\{1,2,\cdots,q\}$ and let $\alpha$ be a variable. Let $V(q,r,\alpha)$ be the set of $r$-tuples $v=(n_1,n_2,\cdots,n_r)$, where all but exactly one entry are in $[q]$ and the  one exceptional entry is $\alpha$.
We call this unique entry the \emph{$\alpha$-bit} of $v$.
Let $\HH_{q}^r$ be the $r$-uniform hypergraph on $V(q,r,\alpha)$ with the edge set $[q]^r$, where a vertex $v=(n_1,n_2,\cdots,n_r)$ with $n_i=\alpha$ is incident to an edge $e$ if substituting $\alpha$ by an integer value in $[q]$ gives the edge $e\in [q]^r$.
Note that $\HH_{q}^r$ is always a linear $r$-graph. Indeed, the codegree of a vertex pair is one if they share all the non-$\alpha$-bits and zero otherwise. 
In particular, if $q=2$, then this is the line hypergraph of the $r$-dimensional hypercube graph.

\begin{proposition}
Let $\HH$ be the auxiliary $r$-graph on $V(H(q))$ whose edge set consists of the standard copies of $H$ in $H(q)$. Then $\HH$ is isomorphic to $\HH_q^r$.
\end{proposition}

\begin{proof}
At the $i$-th iteration of the blow-up procedure, each copy of $v_j$, $j\neq i$, is replaced by $q$ copies of it, each of which is in the edge-disjoint copies of $H_{i-1}$ glued along copies of $v_i$'s. 
By enumerating the $q$ edge-disjoint copies of $H_{i-1}$ in $H_i$, we label each copy of $v_j\in V(H)$ by a vector $(n_1,n_2,\cdots,n_r)\in V(q,r,\alpha)$, where $n_j=\alpha$ and $n_i$, $i\neq j$, indicates that the vertex is in the $n_i$-th copy of $H_{i-1}$ in $H_{i}$.
Let $\phi:V(\HH)\rightarrow V(q,r,\alpha)$ be this labelling map.

We claim that this function $\phi$ is an isomorphism from $\HH$ to $\HH_q^r$.
Indeed, two vertices labelled by $(n_1,n_2,\cdots,n_r)$ and $(m_1,m_2,\cdots,m_r)$, respectively, are in the same standard $H$-copy if and only if $m_i=n_i$ for all $i$ except their $\alpha$-bits.
Hence, $r$ vertices in $V(\HH)$ form an edge if and only if their labels by $\phi$ in $V(q,r,\alpha)$ form an edge in $V(\HH_q^r)$, which proves the claim.
\end{proof}

From now on, we shall identify the $r$-graph $\HH$ with $\HH_{q}^r$.
In a linear hypergraph, a \emph{path} $P$ from a vertex $u$ to another vertex $v$ is an alternating sequence $v_0e_1v_1e_2\cdots v_{\ell-1} e_\ell v_{\ell}$ of vertices and edges, where $v_0=u$, $v_\ell=v$, $\{v_i,v_{i+1}\}\subseteq e_{i+1}$, and any non-consecutive edges are disjoint. 
Two paths  $ue_1v_1e_2\cdots v_{\ell-1}e_\ell v$ and  $ue_1'v_1'e_2'\cdots v_{t-1}'e_t'v$ from $u$ to~$v$ are \emph{internally vertex-disjoint} if $e_1\cap e_1'=\{u\}$, $e_\ell\cap e_t'=\{v\}$, and all the other pairs $e_i$ and $e_j'$ are disjoint.
 
We say that the two paths are \emph{vertex-disjoint} if all edges of one path are disjoint from all edges of the other. Multiple paths $P_1,P_2,\cdots,P_k$ are \emph{(internally) vertex-disjoint} if they are pairwise (internally) vertex-disjoint.
An $r$-graph $\GG$ is \emph{$k$-connected} if there are at least $k$ internally vertex-disjoint paths from a vertex $u$ to another vertex~$v$ for all pairs of distinct vertices $u$ and $v$. We show that $\HH_q^r$ is highly connected in this sense for large enough $q$ in the following proposition, whose proof will be postponed for a while.

\begin{proposition}\label{hypergraph_connected}
For integers $k,r\geq 2$, there exists $q=q_{k,r}$ such that $\HH_q^r$ is $k$-connected.
\end{proposition}
Let $W_i$, $i\in [q]$, and $U_r$ be subsets of $V(q,r,\alpha)$ defined by
\begin{align*}
    W_i:=\{v=(n_1,\cdots,n_r):n_r=i\}~\text{ and }~U_r:=\{u= (m_1,\cdots,m_r): m_r=\alpha\}.
\end{align*}
Let $\HH_i = \HH_q^r[W_i\cup U_r]$ for brevity. The \emph{$r$-extension} of an $(r-1)$-graph $\GG$ is the $r$-graph obtained by adding $e(\GG)$ extra vertices, each of which is added to a unique $(r-1)$-uniform edge in $\GG$. Then $\HH_i$ is isomorphic to the $r$-extension of a copy of the $(r-1)$-graph $\HH_q^{r-1}$ on $W_i$ by the isomorphism that maps each vertex $v=(n_1,\cdots,n_{r-1},i)$ in $\HH_i$ to $(n_1,\cdots,n_{r-1})\in V(q,r-1,\alpha)$ and $(n_1,\cdots,n_{r-1},\alpha)$ to the extra vertex added to extend the edge $(n_1,\cdots,n_{r-1})$.

For vertex subsets $U$ and $V$, a $U$--$V$ path is a path $v_0e_1v_1e_2\cdots v_{\ell-1} e_\ell v_\ell$ such that $v_0\in U$, $v_\ell\in V$, and $v_i\notin U\cup V$ for each $i$ distinct from $0$ and $\ell$.

\begin{lemma}\label{U-U}
For $r\geq 3$ and $1\leq s\leq q$, let $U$ and $V$ be subsets of $U_r$ of size at least $s$. Then there exist vertex-disjoint $U$--$V$ paths $Q_1,Q_2,\cdots,Q_s$ such that each $Q_i$ is a path in $\HH_i$.
\end{lemma}
\begin{proof}
Consider the auxiliary graph $G$ on $U_r$ such that $ww'\in E(G)$ if and only if $w$ and~$w'$ share a neighbour in $\HH_i$. That is, $ww'$ is an edge in $G$ if $w=(n_1,\cdots,n_{r-1},\alpha)$ and $w'=(n_1',\cdots,n_{r-1}',\alpha)$ differ by exactly one entry, which is at the $\alpha$-bit of their common neighbour in $W_i$.
Hence, this graph~$G$ is isomorphic to the graph~$K_{q}^{r-1}$ obtained by taking Cartesian product of $r-1$ copies of $K_{q}$ and moreover, the graph $G$ is independent of the choice of $i\in [q]$.
In particular,~$G$ is $(r-1)(q-1)$-connected, see, e.g., Theorem~1 in~\cite{Spac08}.
By Menger's theorem, there are at least~$s$ vertex-disjoint $U$--$V$ paths in $G$, which we denote by $P_1,P_2,\cdots,P_s$, provided $(r-1)(q-1)\geq q\geq s$. 

\medskip

Our goal is to construct a $U$--$V$ path $Q_i$ in $\HH_i$ by using $P_i$. 
We may assume that $U$ and $V$ are disjoint, as otherwise, one may assign a trivial path at each vertex in the intersection and consider $U':=U\setminus V$ and $V'=V\setminus U$ instead of $U$ and $V$, respectively. It then suffices to find~$s'$ vertex-disjoint $U'$--$V'$ paths, $s'<s$, where induction on $s$ applies.

%Assuming $U\cap V=\emptyset$, %and $e_G(U,V)=0$, %`unused' $r$-graphs $\HH_i$, $i\in [s]$, still remain, both $|U|$ and $|V|$ are at least $s$, and there exist $s$ vertex-disjoint $U$--$V$ paths $P_1,P_2,\cdots,P_{s}$, each of which uses at least two edges.
We choose vertex-disjoint paths $P_1,P_2,\cdots,P_{s}$ that minimise the sum of the length of each path. Then each $P_i$ is an induced path, i.e., there are no $G$-edges on $V(P_i)$ other than $E(P_i)$. To see this, let $P_i=u_0u_1\cdots u_\ell$ with $u_0=u$ and $u_\ell=v$.
If there is an edge $u_iu_j$ with $i+1<j$, then one can shorten the length of $P_i$ by replacing the path $u_iu_{i+1}\cdots u_j$ by $u_iu_j$. 
The internal vertices of the shorter path $P_i'$ is still non-empty as $uv\notin E(G)$ and disjoint from the internal vertices of other $P_j$'s, so we strictly reduce the sum of the $s$  vertex-disjoint paths.

Now each $U$--$V$ path $P_i=u_{0}u_{1}u_{2}\cdots u_{\ell}$ yields a $U$--$V$ path $Q_i$ in~$\HH_i$. 
Indeed, there exists a unique edge $e_{j}\in \HH_i$ containing $u_{j}$ such that $e_{j}$ and $e_{j+1}$ share a vertex~$w_{j}$ in $W_i$ by definition of~$G$.
Furthermore, two non-consecutive edges $e_j$ and $e_{j'}$, $j+1<j'$, are always disjoint, as otherwise $u_ju_{j'}\in E(G)$. Therefore, $u_0e_1u_1e_2\cdots u_{\ell-1}e_{\ell}u_\ell$ is a path in $\HH_i$. 
It then remains to check whether $Q_1,\cdots,Q_k$ in $\HH_q^r$ are vertex-disjoint. The vertices in $Q_i$ and $Q_j$ are in $W_i\cup U_r$ and $W_j\cup U_r$, respectively. Indeed, the two sets $W_i$ and $W_j$ are disjoint and the vertices of $Q_i$ and $Q_j$ in $U_r$ are disjoint too, as they are exactly vertices of $P_i$ and $P_j$, respectively.
\end{proof}

\begin{proof}[Proof of \Cref{hypergraph_connected}]
If $r=2$ then $\HH_q^r$ is a copy of $K_{q,q}$, which is $q$-connected. We may hence assume that $r\geq 3$. Take $q\geq \max\{q_{r-1,k}, 3(k+1)\}$ which is a multiple of $3$.
Let $u,v$ be distinct vertices in $\HH_q^r$. 
By induction on $r$, $\HH_q^{r-1}$ is $k$-connected. As $\HH_i$ is the $r$-extension of $\HH_q^{r-1}$, there are at least $k$ internally vertex-disjoint paths in $\HH_i$ from $u$ to~$v$ if both vertices are in $W_i$. 

Suppose that $u,v\in U_r$. For $1\leq i\leq q/3$, let $P_i$ be the path $ue_{i,1}w_ie_{i,2}u_i$ in $\HH_i$, i.e., $w_i\in W_i$, $u_i\in U_r$, and $e_{i,1}$ is the only edge in $\HH_i$ containing $u$. 
We may further assume that all~$u_i$'s are distinct, as there are $q-1$ neighbours of $w_j$ in $U_r$ except $u$. Analogously, take paths $P_j'=ve_{j,1}w_je_{j,2}v_j$ for $q/3<j\leq 2q/3$ in $\HH_j$ where $v_j$'s are all distinct.
Applying \Cref{U-U} with $U=\{u_i:1\leq i\leq q/3\}$, $V=\{v_j:q/3<j\leq 2q/3\}$, and $s=q/3$ gives $q/3$ vertex-disjoint $U$--$V$ paths, each of which uses a unique $\HH_t$ for some $t>2q/3$. Here we relabel $\HH_i$'s if necessary. Thus, concatenating these $U$--$V$ paths with $P_i$ and $P_j'$ yields at least $k$ internally vertex-disjoint paths from $u$ to $v$.  

Next, suppose that $u\in U_r$ and $v\in W_j$. For all $i\in[q/3]\setminus\{j\}$, we analogously collect paths $P_i$ of length two from $u$ such that each $P_i$ is in $\HH_i$ and ends at $u_i\in U_r$, where $u_i$'s are all distinct. There are $q$ neighbours of $v$ in $U_r$, which we denote by $N(v;U_r)$. 
Then again by \Cref{U-U}, there are at least $q/3-1$ vertex-disjoint $U$--$V$ paths from $U=\{u_i:i\in[q/3]\setminus\{j\}\}$ to $V=N(v;U_r)$, each of which uses distinct $\HH_t$ such that $t\neq j$ and $t>q/3$. Concatenating these $U$--$V$ paths with $P_i$'s and the edges incident to $v$ gives $k$ internally vertex-disjoint paths from $u$ to $v$.  

Lastly, suppose that $u\in W_i$ and $v\in W_j$ for $i\neq j$. Let $N(u;U_r)$ and $N(v;U_r)$ be neighbours of $u$ and $v$ in $U_r$, respectively. 
Then \Cref{U-U} gives $q$ vertex-disjoint $N(u;U_r)$--$N(v;U_r)$ paths. After deleting those paths in $\HH_i$ or $\HH_j$, there are still at least $k$ paths left, which allow us to make~$k$ internally vertex-disjoint path from $u$ to $v$. 
\end{proof}

\Cref{blow_up_connected} follows from the fact that internally disjoint paths in $\HH_q^r$ translate to internally disjoint paths in $H(q)$.

\begin{lemma}
Let $H$ be a connected graph and let $P_1, P_2,\cdots,P_k$ be $k$ internally vertex-disjoint paths from $u$ to $v$ in $\HH_q^r$.
Then there exist internally vertex-disjoint paths $Q_1,Q_2,\cdots Q_k$ from $u$ to $v$ in $H(q)$ such that each $Q_i$, $i\in[k]$, only uses those edges and vertices in the standard $H$-copies that correspond to edges in $P_i$.
\end{lemma}

\begin{proof}
Let $P_i = v_{i,0}e_{i,1}v_{i,1}e_{i,2}\cdots v_{i,\ell_{i-1}}e_{i,\ell_i} v_{i,\ell_i}$.
As $H$ is connected, there exists a path $Q_{i,j+1}$ from~$v_{i,j}$ to $v_{i,j+1}$ in the standard copy of $H$ that corresponds to the edge $e_{i,j+1}$.
For paths $P$ from $x$ to~$y$ and~$P'$ from $y$ to $z$, we write $xPyP'z$ for the concatenation of the two paths from $x$ to $z$.
For each $i\in[k]$,
let $Q_i=v_{i,0}Q_{i,1}v_{i,1}Q_{i,2}\cdots v_{i,\ell_{i-1}}Q_{i,\ell_i}v_{i,\ell_i}$.
We claim that these paths $Q_1,Q_2,\cdots,Q_k$ are internally vertex-disjoint. Indeed, $Q_{i,1}$ and $Q_{i',1}$ are in the $H$-copies that correspond to $e_{i,1}$ and~$e_{i',1}$, respectively, who share the vertex $u=v_{i,0}=v_{i',0}$ only; by the same reason, $Q_{i,\ell_i}$ and $Q_{i',\ell_{i'}}$ are also disjoint except the vertex $v=v_{i,\ell_i}=v_{i',\ell_{i'}}$;
the other $Q_{i,j}$ and $Q_{i',j'}$ are vertex-disjoint, since~$e_{i,j}$ and $e_{i',j'}$ are disjoint edges in $\HH_q^r$.
\end{proof}

\section{Concluding remarks}

After~\Cref{main}, it would be natural to ask for examples of common graphs that are even more challenging to find.
We suggest to find a common graph with arbitrarily large girth, chromatic number, and connectivity. 

\begin{question}\label{openq}
Let $r,k,g\geq 3$ be integers.
Does there exist an $r$-chromatic $k$-connected common graph with girth at least $g$?
\end{question}

We believe that such a common graph exists;
however, our blown-up graph $H(q)$ in~\Cref{main} may decrease the girth of $H$, as the construction produces 4-cycles whenever $q\geq 2$ and $E(H)$ is nonempty. This suggests that solving~\Cref{openq} might require new ideas.

\vspace{5mm}

\noindent\textbf{Acknowledgements.} This work is supported by the National Research Foundation of Korea (NRF) grant funded by the Korea government MSIT NRF-2022R1C1C1010300 and Samsung STF Grant SSTF-BA2201-02. The second author was also supported by IBS-R029-C4. We would like to thank Dan Kr\'a\v{l}, Jan Volec, and Fan Wei for helpful discussions.
We are also grateful to anonymous referees for their careful reviews.

\bibliographystyle{plainurl}
\bibliography{references}

\end{document}